\documentclass[runningheads]{llncs}
\usepackage{amssymb}
\usepackage{graphicx}
\usepackage[extra]{tipa}
\usepackage{amsmath}
\usepackage{url}
\usepackage{color}
\urldef{\mailsa}\path| idimit@math.upatras.gr|
\newcommand{\keywords}[1]{\par\addvspace\baselineskip
\noindent\keywordname\enspace\ignorespaces#1}

\begin{document}
\title{Stationary analysis of a tandem queue with coupled processors subject to global breakdowns}
\titlerunning{A tandem queue with coupled processors subject to global breakdowns}
%
\author{Ioannis Dimitriou\orcidID{0000-0002-3755-9698}}
\authorrunning{I. Dimitriou}
%
\institute{Department of Mathematics, University of Patras, 26500 Patras, Greece, \mailsa\\
\url{https://thalis.math.upatras.gr/~idimit/}}
\maketitle              
\begin{abstract}
We consider a tandem queue with coupled processors, which is subject to global breakdowns. When the network is in the operating mode and both queues are non empty, the total service capacity is shared among the stations according to fixed proportions. When one of the stations becomes empty, the total service capacity is given to the non-empty station. Moreover, arrival rates depend on the state of the network. The system is described by a Markov modulated random walk in the quarter plane representing the number of jobs in the two stations and the state of the network. By applying the generating function approach, we first apply the power series approximation method to obtain power series expansions of the generating function of the stationary queue lengths for both network states. Then, we also provide a way to derive the generating function of the stationary queue lengths for both network states in terms of the solution of a Riemann-Hilbert boundary value problem. Numerical results are obtained to show insights in the system performance.

\end{abstract}
\section{Introduction}
Queueing networks with service interruptions are known to be adequate models to handle realistic problems in manufacturing, telecommunications etc. Despite their great importance, there has been done very few works involving more
than one queue, since even under favorable assumptions, the
existence of service interruptions destroys the separability \cite{bcmp} of the appropriate multidimensional Markov process and renders its solution intractable. 

The problem becomes even more challenging when we further assume that the nodes are interacting with each other. With the term ``interaction", we mean that the service rate at a node depends on the state of the other nodes, i.e., a networks with coupled processors \cite{fay1,res,van}. For such systems, which do not possess the ``product form" solution, analytic methods have been developed in \cite{fay,bv}. 

In this work, we go one step further and consider a two-node tandem network with coupled processors, which is subject to global (i.e., network) breakdowns. It is assumed that jobs arrive at the first
station according to a Poisson process according to the state of the network, and require service at both stations before
leaving the system. The amounts of work that a job requires at each of the stations
are independent, exponentially distributed random variables. When the network is in the operating mode, and both stations
are nonempty, the total service capacity is shared between the stations according
to fixed proportions. 

When one of the stations becomes empty, the total service
capacity is given to the non-empty station. When the network is in the setup mode after a failure occurrence, both stations stop working\footnote{This is natural when we are dealing with queues in series.}, for an exponentially distributed time period. During setup period, jobs continue to arrive, but now at a decreased rate in order to avoid further congestion\footnote{Such an operation can be performed by a central scheduler, responsible for the congestion management.}. Such a system is fully described by a random walk in the quarter plane (representing the number of jobs in each node), which is modulated by a two state Markov chain (representing the state of the network). 

For such a network we provide to different approaches, named the Power Series Approximation (PSA) method, and the theory of Riemann-Hilbert boundary value problems (BVP) to investigate its stationary behavior.
\subsection{Related work}
Most of the existing studies involving breakdowns have concentrated on models with
a single job queue served by one or more processors, e.g. \cite{tak,mitrr}. Other related results where jobs from a single source are directed to one of
several parallel queues, and where breakdowns result in the loss of jobs, or the direct transfer to other queues are given in \cite{thom,idri,mitr}. A network of two nodes subject to breakdowns and repairs was also analysed in \cite{mik}. Approximate solutions to obtain performance measures, based on replacing an interruptable server with an uninterruptable but slower one, choosing the new service rate without affecting the overall service
capacity was given in \cite{rei,vin}.

Queues with coupled processors were initially studied in \cite{fay1}. To gain quantitative insights about the queueing process, they derived a solution for the generating function of the stationary joint queue-length distribution using the theory of Riemann-Hilbert boundary value problems. Later, in \cite{bv} a systematic and detailed study of the technique of reducing a two dimensional functional equation of a random walk or queueing model to a boundary
value problem was presented, while several numerical issues were discussed. Important generalizations were given in \cite{coh,avr,gui1,van,res,dim1,dim2,dimpap,dimpap1,dimpaptwc,dimpapad} (not exhaustive list) where various two-dimensional queueing models with the aid of the theory of Riemann (-Hilbert) boundary value problems. 

A tandem queue with two coupled processors was analyzed in \cite{res}, while later computational issues, as well as asymptotic results were discussed and presented in \cite{van,gui2}, respectively. There, it was shown that the problem of finding the bivariate generating function of the stationary joint queue length distribution is reduced to a Riemann-Hilbert boundary value problem.

Applications of coupled processor models arise systems where limited resources are dynamically shared among processors, e.g., in data transfer in bidirectional cable and
data networks \cite{van}, in bandwidth sharing of data flows \cite{gui2}, in the performance modeling of device-to-device communication \cite{vita}, to model the complex interdependence among transmitters due to interference \cite{borst0,dimpaptwc,dimpap,dimpapad}, as well as in assembly lines in manufacturing \cite{antr}.

Other approaches to analyze two-dimensional queueing models have been developed in \cite{ad} (compensation method), and \cite{blanc2,pss} (power series algorithm). In the latter one, power series expansions of steady-state probabilities as functions of a certain parameter of the system (usually the load) were derived. Recently, the authors in \cite{walr} studied generalized processor sharing queues and introduced an alternative method, by providing power series expansions of the generating function; see also \cite{vanle,vanle1,dim}.
\subsection{Our contribution}
In this work we focus on the stationary analysis of two-node tandem queue with coupled processors and network breakdowns.  
\paragraph{Applications} Potential application of our system are found in systems with limited capacity, which must be shared in multiple operations. For example, in the modeling of virus attacks or other malfunctions in cable access networks regulated by a request-grant mechanism. 

Another application of the model can be found in manufacturing \cite{antr}, and in particular in an assembly line. There, two operations on each job must be performed using a limited service capacity. To increase the network throughput, we couple the service rates at each of the operations, and thus we use the service capacity of an operation for which no jobs are waiting for the other operation. Such a system is heavily affected by the presence of failures during the job processing, which in turn will definitely deteriorate the system throughput.
\paragraph{Fundamental contribution} Based on the generating function approach, we apply two different methods to investigate the stationary behaviour of the underline Markov modulated random walk in the quarter plane. First, we apply the power series approximation method, initially introduced in \cite{walr} (see also \cite{walra,vanle1,vanle}), for two-parallel generalized processor sharing queues, described by a typical random walk in the quarter plane (RWQP). In this work we show that this method is still valid for related \textit{Markov modulated} RWQP, and thus, extend the class of models that can be applied. Under such a method we obtain power series expansions of the probability generating function (pgf) of the joint stationary distribution for either state of the network. A recursive technique to derive their coefficients is also presented.

Secondly, we also derive the pgfs of the joint stationary queue-length distribution for either state of the network with the aid of the theory of Riemann-Hilbert boundary value problems \cite{fay1,fay,bv}. It is seen that applying the theory of boundary value problems some further technical difficulties are also arise. More precisely, by applying the generating function approach we first come up with a system of functional equations, which is then reduced to a single fundamental equation.

The rest of the paper is summarized as follows. In Section \ref{sec:mod} we present the mathematical model in detail and obtain the functional equations along with some preliminary results. Section \ref{sec:sp} is devoted to the analysis of the two extreme cases where the total capacity is allocated to one of the two stations, even if both stations are nonempty. In Section \ref{sec:psa} we apply the power series approximation method for the case where the service capacity is shared by the two stations, where in Section \ref{sec:bvp}, we provide a complete analysis on how to obtain the pgfs of the stationary joint queue length distribution in terms of a solution of a Riemann-Hilbert boundary value problem. Numerical validations of the performance metrics obtained by using the PSA and the BVP approaches are given in \ref{sec:num}.
\section{The model and the functional equations}\label{sec:mod}
Consider a two-stage tandem queue, where jobs arrive at queue 1 according to a Poisson
process with rate depending on the state of the network. In particular, the network is subject to breakdowns which occur according to a Poisson process with rate $\gamma$. When a breakdown occurs, both stations stop working for an exponentially distributed time period with rate $\tau$. Thus, the network alternates between the \textit{operating mode} and the \textit{setup mode}. Denote by $C(t)$ the state of the network at time $t$, with $C(t)=0$ (resp. 1), when network is in operating (resp. setup). When $C(t)=i$, jobs arrive in station 1 according to a Poisson process with rate $\lambda_{i}$, $i=0,1$. 

Each job demands service at both queues before departing from the
network. More precisely, at station $j$, a job requires an exponentially distributed amount of service with parameter $\nu_{j}$, $j = 1, 2$. The total service capacity of the tandem network equals one unit of
work per time unit. In particular, when both stations are non-empty, station $j$ is served at a rate $\phi_{j}$, $j=1,2$, and at a rate 1 when it is the only non empty. Without loss of generality we assume hereon that $\phi_{1}=p$ and $\phi_{2}=1-p$, where $0\leq p\leq1$.

Let $Q_{j}(t)$, $j=1,2,$ be the number of customers at queue $j$ at time $t$. Under usual assumptions the stochastic process $X(t)=\{(C(t),Q_{1}(t),Q_{2}(t));t\geq0\}$ is an irreducible and aperiodic continuous time Markov chain with state space $E=\{0,1\}\times\mathbb{Z}^{+}\times\mathbb{Z}^{+}$. Denote by $\pi_{i}(n,k)$ the stationary probability of having $n$ and $k$ customers at stations 1 and 2, respectively, when the network is in state $i$. The balance equations are given by
\begin{equation}
\begin{array}{rl}
(\lambda_{0}+\gamma)\pi_{0}(0,0)=&\nu_{2}\pi_{0}(0,1)+\tau\pi_{1}(0,0),\\
(\lambda_{0}+\gamma+\nu_{2})\pi_{0}(0,1)=&\nu_{1}\pi_{0}(1,0)+\nu_{2}\pi_{0}(0,2)+\tau\pi_{1}(0,1),\\
(\lambda_{0}+\gamma+\nu_{2})\pi_{0}(0,k)=&p\nu_{1}\pi_{0}(1,k-1)+\nu_{2}\pi_{0}(0,k+1)+\tau\pi_{1}(0,k),\,k\geq2\\
\end{array}
\end{equation}
\begin{equation}
\begin{array}{l}
(\lambda_{0}+\gamma+\nu_{1})\pi_{0}(n,0)=\lambda_{0}\pi_{0}(n-1,0)+(1-p)\nu_{2}\pi_{0}(n,1)+\tau\pi_{1}(n,0),\vspace{2mm}\\
(\lambda_{0}+\gamma+p\nu_{1}+(1-p)\nu_{2})\pi_{0}(n,1)=\lambda_{0}\pi_{0}(n-1,1)+\nu_{1}\pi_{0}(n+1,0)\\+(1-p)\nu_{2}\pi_{0}(n,2)+\tau\pi_{1}(n,1),\,n\geq1,\vspace{2mm}\\
(\lambda_{0}+\gamma+p\nu_{1}+(1-p)\nu_{2})\pi_{0}(n,k)=\lambda_{0}\pi_{0}(n-1,k)+p\nu_{1}\pi_{0}(n+1,k-1)\\+(1-p)\nu_{2}\pi_{0}(n,k+1)+\tau\pi_{1}(n,k),\,n\geq1,k\geq2\end{array}
\end{equation}
\begin{equation}
\begin{array}{c}
(\lambda_{1}+\tau)\pi_{1}(n,k)=\gamma\pi_{0}(n,k)+\lambda_{1}\pi_{1}(n-1,k).
\end{array}
\end{equation}
Define the probability generating functions of the joint stationary queue length distribution 
\begin{displaymath}
\Pi_{i}(x,y)=\sum_{n=0}^{\infty}\sum_{k=0}^{\infty}\pi_{i}(n,k)x^{n}y^{k},\,i=0,1,\,|x|\leq1,|y|\leq1.
\end{displaymath}

Using the balance equations we obtain after some algebra the following system of functional equations
\begin{equation}
\begin{array}{rl}
R(x,y)\Pi_{0}(x,y)=&A(x,y)\Pi_{0}(x,0)+B(x,y)\Pi_{0}(0,y)\\&+C(x,y)\Pi_{0}(0,0)+\tau xy\Pi_{1}(x,y),\vspace{2mm}\\
\
\Pi_{1}(x,y)=&\frac{\gamma}{D(x)}\Pi_{0}(x,y),
\end{array}
\label{f1}
\end{equation}
where, $D(x)=\lambda_{1}(1-x)+\tau$ and
\begin{equation}
\begin{array}{rl}
R(x,y)=&xy(\lambda_{0}(1-x)+\gamma)+\nu_{1}py(x-y)+\nu_{2}(1-p)x(y-1),\\
A(x,y)=&(1-p)[\nu_{2}x(y-1)+\nu_{1}y(y-x)],\\
B(x,y)=&-\frac{p}{1-p}A(x,y),\\
C(x,y)=&\nu_{1}(1-p)y(x-y)+\nu_{2}px(y-1).
\end{array}\label{ui}
\end{equation}

Our aim is to solve the system of functional equations (\ref{f1}). Substituting the second in (\ref{f1}) to the first one, we obtain the following fundamental functional equation
\begin{equation}
\begin{array}{c}
\Pi_{0}(x,y)[D(x)R(x,y)-\tau\gamma xy]=D(x)\{A(x,y)\Pi_{0}(x,0)+B(x,y)\Pi_{0}(0,y)\\+C(x,y)\Pi_{0}(0,0)\}.
\end{array}
\label{f2}
\end{equation}
Clearly, $\Pi_{0}(1,1)+\Pi_{1}(1,1)=1$, while by using the second in (\ref{f1}) we obtain $\Pi_{1}(1,1)=\frac{\tau}{\gamma}\Pi_{0}(1,1)$. Thus, the probabilities of the network state are easily given by
\begin{displaymath}
\Pi_{0}(1,1)=\frac{\tau}{\tau+\gamma},\,\Pi_{1}(1,1)=\frac{\gamma}{\tau+\gamma}.
\end{displaymath}

Let $x=s(y):=\frac{\nu_{1}y^{2}}{\nu_{1}+\nu_{2}(1-y)}$. Note that $A(s(y),y)=B(s(y),y)=0$. Then, using (\ref{f2}) we obtain
\begin{equation}
\Pi_{0}(s(y),y)=\frac{D(s(y))C(s(y),y)}{D(s(y))R(s(y),y)-\tau\gamma s(y)y}\Pi_{0}(0,0).
\label{fg}
\end{equation}
Letting $y\to 1$ in (\ref{fg}) we obtain
\begin{equation}
\Pi_{0}(0,0)=\frac{\tau}{\tau+\gamma}-\left(\lambda_{0}\frac{\tau}{\tau+\gamma}+\lambda_{1}\frac{\gamma}{\tau+\gamma}\right)\left(\frac{1}{\nu_{1}}+\frac{1}{\nu_{2}}\right).
\label{hj}
\end{equation}
Note that (\ref{hj}) implies that our network is stable when
\begin{equation}
\lambda_{0}\left(\frac{1}{\nu_{1}}+\frac{1}{\nu_{2}}\right)\frac{\tau}{\tau+\gamma}+\lambda_{1}\left(\frac{1}{\nu_{1}}+\frac{1}{\nu_{2}}\right)\frac{\gamma}{\tau+\gamma}<\frac{\tau}{\tau+\gamma},
\label{stab}
\end{equation}
which can be explained by realizing that the left hand side of (\ref{stab}) equals the amount of work brought into the system per time unit, and in order the system to be stable, should be less than the amount of work departing the system per time unit. 

Let $\rho_{kj}:=\frac{\lambda_{k}}{\nu_{j}}$, $k=0,1$, $j=1,2$, and $\rho_{k}=\rho_{k1}+\rho_{k2}$, $k=1,2$. Then, (\ref{hj}) is rewritten as
\begin{displaymath}
\Pi_{0}(0,0)=\frac{\tau}{\tau+\gamma}(1-\frac{\rho_{0}\tau+\rho_{1}\gamma}{\tau}).
\end{displaymath}
\begin{remark}
Note that $\lambda_{0}\left(\frac{1}{\nu_{1}}+\frac{1}{\nu_{2}}\right)\frac{\tau}{\tau+\gamma}$ (resp. $\lambda_{1}\left(\frac{1}{\nu_{1}}+\frac{1}{\nu_{2}}\right)\frac{\gamma}{\tau+\gamma}$) refers to the amount of work that arrive at the system per time unit when the network is in the operating mode (resp. in the setup mode), while a job can depart from the network only when it is in the operating mode, and this is happening with probability $\tau/(\tau+\gamma)$.
\end{remark}
\section{The cases $p=0$ and 1}\label{sec:sp}
When $p=0$ (resp. $p=1$), the model can be seen as a tandem queues served by a single server, in which preemptive
priority is given to station 2 (resp. station 1). It is easily seen that in such cases, the functional equation (\ref{f1}) can be easily solved since either the coefficient of $\Pi_{0}(0,y)$ (when $p=0$), or the one of $\Pi_{0}(x,0)$ (when $p=1$) is equal to zero. 

In case $p= 0$ (i.e., $B(x,y)=0$), upon a service completion in station 1, the server continues serving the customer in station 2, since station 2 has priority. Thus, in such a case our system reduces to an unreliable queueing system, in which the service time
consists of two exponential phases with parameters
$\nu_{1}$ and $\nu_{2}$, respectively. Note that for $p=0$,
\begin{displaymath}
y:=\xi(x)=\frac{\nu_{2}xD(x)}{xD(x)(\nu_{2}+\lambda_{0}(1-x))+\lambda_{1}\gamma x(1-x)},
\end{displaymath}
vanishes the left-hand side of (\ref{f1}), and yields
\begin{displaymath}
\Pi_{0}(x,0)=\left(\frac{\tau}{\tau+\gamma}\right)\frac{(1-\frac{\rho_{0}\tau+\rho_{1}\gamma}{\tau})\nu_{1}\xi(x)(\xi(x)-x)}{\nu_{2}x(\xi(x)-1)+\nu_{1}\xi(x)(x-\xi(x))}.
\end{displaymath}
Substituting back in (\ref{f1}) yields
\begin{equation}
\begin{array}{rl}
\Pi_{0}(x,y)=&\frac{(\frac{\tau}{\tau+\gamma})(1-\frac{\rho_{0}\tau+\rho_{1}\gamma}{\tau})D(x)\nu_{1}\nu_{2}}{D(x)[\lambda_{0}y(1-x)+\nu_{2}(y-1)]+\lambda_{1}\gamma y(1-x)}\vspace{2mm}\\&\times\left\{\frac{\xi(x)(y-1)(\xi(x)-x)+y(x-y)\nu_{2}(\xi(x)-1)}{\nu_{2}x(\xi(x)-1)+\nu_{1}\xi(x)(x-\xi(x))}\right\},\vspace{2mm}\\
\Pi_{1}(x,y)=&\frac{\gamma}{D(x)}\Pi_{0}(x,y).\end{array}
\end{equation}

The case $p=1$ is even more interesting and corresponds to an unreliable tandem queue attended by a single server and preemptive priority for the first station. That is, if upon a customer arrival the server is at the second station, it switches immediately to the first one. Moreover, upon a setup completion, after a global breakdown, the server will start serving at the first station if there are customers waiting. And this is the case even if a breakdown occurs when was serving a customer at the second station. To the author's best knowledge, that case has never considered before. For $p=1$, (\ref{f1}) reduces to
\begin{equation}
\begin{array}{l}
[D(x)(\lambda_{0}x(1-x)+\nu_{1}(x-y))+\lambda_{1}\gamma x(1-x)]y\Pi_{0}(x,y)=\vspace{2mm}\\D(x)\{\Pi_{0}(0,y)[\nu_{2}x(1-y)+\nu_{1}y(x-y)]+\Pi_{0}(0,0)\nu_{2}x(y-1)\}.
\end{array}
\label{sxx}
\end{equation}
Let $x:=u(y)$ the unique root of $D(x)(\lambda_{0}x(1-x)+\nu_{1}(x-y))+\lambda_{1}\gamma x(1-x)=0$ inside the unit circle. Then, the right-hand side should also vanish and thus,
\begin{displaymath}
\Pi_{0}(0,0)=\Pi_{0}(0,y)[1-\frac{\nu_{1}y(u(y)-y)}{\nu_{2}u(y)(1-y)}].
\end{displaymath}
Substituting back in (\ref{sxx}) yields
\begin{equation}
\begin{array}{rl}
\Pi_{0}(x,y)=&\frac{D(x)\nu_{1}y(x-u(y))}{u(y)[D(x)(\lambda_{0}x(1-x)+\nu_{1}(x-y))+\lambda_{1}\gamma x(1-x)]}\Pi_{0}(0,y),\vspace{2mm}\\
\Pi_{1}(x,y)=&\frac{\gamma}{D(x)}\Pi_{0}(x,y).
\end{array}
\end{equation}
\section{The case $0<p<1$: Power Series Approximation in $p$}\label{sec:psa}
In the following, we are going to construct a power series expansion of the pgf $\Pi_{0}(x,y)$ in $p$ starting by (\ref{f2}). Then, having that result we are able to construct power series expansions of $\Pi_{1}(x,y)$ in $p$ using the second in (\ref{f1}). With that in mind, let
\begin{equation}
\Pi_{j}(x,y)=\sum_{m=0}^{\infty}V_{m}^{(j)}(x,y)p^{m},\,j=0,1.
\end{equation} 
Our aim in the following, is to obtain $V_{m}^{(j)}(x,y)$, $m\geq0$, $j=0,1$, by employing an approach similar to the one developed in \cite{walra,walr,dim}\footnote{See Appendix \ref{appe3} for the analyticity of $\Pi_{j}(x,y)$ close to $p=0$.}. Equation (\ref{f2}) is rewritten as
\begin{equation}
\begin{array}{c}
G(x,y)\Pi_{0}(x,y)-G_{10}(x,y)\Pi_{0}(x,0)-G_{00}(x,y)\Pi_{0}(0,0)\vspace{2mm}\\
=pG_{10}(x,y)[\Pi_{0}(x,y)-\Pi_{0}(x,0)-\Pi_{0}(0,y)+\Pi_{0}(0,0)],
\end{array}
\label{f4}
\end{equation}
where,
\begin{displaymath}
\begin{array}{rl}
G(x,y)=&D(x)[\lambda_{0}y(1-x)+\nu_{2}(y-1)]+\lambda_{1}\gamma y(1-x),\\
G_{10}(x,y)=&D(x)[\nu_{2}(y-1)-\nu_{1}y(1-yx^{-1})],\\
G_{00}(x,y)=&D(x)\nu_{1}y(1-yx^{-1}).
\end{array}
\end{displaymath}

The major difficulty in solving (\ref{f2}) corresponds to the presence of the two unknown boundary
functions $\Pi_{0}(x,0)$, $\Pi_{0}(0,y)$. Having in mind that in the left-hand side of (\ref{f4}) there is only one boundary function, we are able to follow the approach in \cite{walr,walra,dim}. The next Theorem summarizes the basic result of this section.
\begin{theorem}\label{th}
Under stability condition (\ref{stab}),
\begin{equation}
\begin{array}{rl}
V_{0}^{(0)}(x,y)=&\frac{D(x)\nu_{1}\nu_{2}(\tilde{Y}(x)-x)(\tilde{Y}(x)(y-1)+x-y)}{G(x,y)[\nu_{2}x(\tilde{Y}(x)-1)-\nu_{1}\tilde{Y}(x)(x-\tilde{Y}(x))]}V^{(0)}(0,0),\vspace{2mm}\\
V_{m}^{(0)}(x,y)=&\frac{G_{10}(x,y)}{G(x,y)}Q_{m-1}(x,y),\,m>0,\vspace{2mm}\\
V_{m}^{(1)}(x,y)=&\frac{\gamma}{D(x)}V_{m}^{(0)}(x,y),\,m\geq0,
\end{array}
\label{e0}
\end{equation}
where 
\begin{displaymath}
\begin{array}{rl}
\tilde{Y}(x)=&\frac{\nu_{2}D(x)}{D(x)(\nu_{2}+\lambda_{0}(1-x))+\lambda_{1}\gamma(1-x)},\\
Q_{m}(x,y)=&V_{m}^{(0)}(x,y)-V_{m}^{(0)}(x,\tilde{Y}(x))-V_{m}^{(0)}(0,y)+V_{m}^{(0)}(0,\tilde{Y}(x)),\,m\geq0,
\end{array}
\end{displaymath}
and $Q_{-1}(x,y):=0$.
\end{theorem}
\begin{proof}
The proof follows the lines in \cite{walr,dim}. Note that $\Pi_{0}(x,y)$ is analytic function of $p$ in a neighbourhood of $0$. We start by expressing $\Pi_{0}(x,y)$ in power series expansion of $p$ by using (\ref{f4}), and equating the corresponding powers of $p$ at both sides. This yields
\begin{equation}
\begin{array}{rr}
V_{m}^{(0)}(x,y)G(x,y)=&G_{10}(x,y)[V_{m}^{(0)}(x,0)+P_{m-1}(x,y)]\\&+G_{00}(x,y)V_{m}^{(0)}(0,0),\,m\geq0,
\end{array}
\label{e1}
\end{equation}
where
\begin{displaymath}
P_{m}(x,y)=V_{m}^{(0)}(x,y)-V_{m}^{(0)}(x,0)-V_{m}^{(0)}(0,y)+V_{m}^{(0)}(0,0),
\end{displaymath}
with $P_{-1}(x,y)=0$. Note that $G(x,y)=0$ has a unique zero $y=\tilde{Y}(x)$ such that
\begin{displaymath}
\tilde{Y}(x)=\frac{\nu_{2}D(x)}{D(x)(\nu_{2}+\lambda_{0}(1-x))+\lambda_{1}\gamma(1-x)}.
\end{displaymath}
It is easy to realize that $|\tilde{Y}(x)|<1$, for $|x|=1$. Substituting in (\ref{e1}) we eliminate its left-hand side yielding
\begin{equation}
V_{m}^{(0)}(x,0)=-\frac{G_{00}(x,\tilde{Y}(x))}{G_{10}(x,\tilde{Y}(x))}V_{m}^{(0)}(0,0)-P_{m-1}(x,y).
\label{e2}
\end{equation}
Substituting (\ref{e2}) back in (\ref{e1}) yields the first in (\ref{e0}). Then, using the second equation in (\ref{f1}) we derive the coefficients $V_{m}^{(1)}(x,y)$ in terms of $V_{m}^{(0)}(x,y)$ as given in the second in (\ref{e0}). From (\ref{hj}) it is readily seen that
\begin{displaymath}
\begin{array}{rl}
V_{0}^{(0)}(0,0)=&\frac{\tau}{\tau+\gamma}-\left(\lambda_{0}\frac{\tau}{\tau+\gamma}+\lambda_{1}\frac{\gamma}{\tau+\gamma}\right)\left(\frac{1}{\nu_{1}}+\frac{1}{\nu_{2}}\right),\\
V_{m}^{(0)}(0,0)=&0,\,m>0.
\end{array}
\end{displaymath}
\end{proof}
\subsection{Performance metrics}
We focus on the mean queue lengths given by
\begin{equation}
\begin{array}{rl}
E(Q_{1})=&\sum_{m=0}^{\infty}p^{m}\frac{\partial}{\partial x}[V_{m}^{(0)}(x,1)+V_{m}^{(1)}(x,1)]|_{x=1}\\=&\frac{\lambda_{1}\gamma}{\tau(\tau+\gamma)}+(1+\frac{\gamma}{\tau})\sum_{m=0}^{\infty}p^{m}\frac{\partial}{\partial x}V_{m}^{(0)}(x,1)|_{x=1},\vspace{2mm}\\
E(Q_{2})=&\sum_{m=0}^{\infty}p^{m}\frac{\partial}{\partial y}[V_{m}^{(0)}(1,y)+V_{m}^{(1)}(1,y)]|_{y=1}\vspace{2mm}\\=&(1+\frac{\gamma}{\tau})\sum_{m=0}^{\infty}p^{m}\frac{\partial}{\partial y}V_{m}^{(0)}(1,y)|_{y=1},
\end{array}
\label{per}
\end{equation}
Let
\begin{displaymath}
v_{m,1}=\frac{\partial}{\partial x}V_{m}^{(0)}(x,1)|_{x=1},\,v_{m,2}=\frac{\partial}{\partial y}V_{m}^{(0)}(1,y)|_{y=1}.
\end{displaymath}
Truncation of the power series in (\ref{per}) yields,
\begin{equation}
\begin{array}{rl}
E(Q_{1})=&\frac{\lambda_{1}\gamma}{\tau(\tau+\gamma)}+(1+\frac{\gamma}{\tau})\sum_{m=0}^{M}p^{m}v_{m,1}+O(p^{M+1}),\\
E(Q_{2})=&(1+\frac{\gamma}{\tau})\sum_{m=0}^{\infty}p^{m}v_{m,2}+O(p^{M+1}).
\end{array}
\label{pert}
\end{equation}
Truncation yields accurate approximations for $p$ close to 0. However, we have to note the actual calculation of the expressions in
(\ref{per}) requires the computation of the first derivatives of $V_{m}^{(0)}(x,y)$ for $m\geq 0$. Although we provided an algorithm (see Theorem \ref{th}) to calculate these coefficients, the calculation of their first derivatives is far from straightforward due to the extensive use of L’Hopital’s rule.
\section{The case $0<p<1$: A Riemann-Hilbert Boundary Value Problem}\label{sec:bvp}
In the following, we proceed with the determination of $\Pi_{j}(x,y)$, $j=0,1,$ $|x|\leq1$, $|y|\leq1$ with the aid of the theory of boundary value problems (BVP). In particular, we first obtain $\Pi_{0}(x,y)$ in terms of the solution of a Riemann-Hilbert boundary value problem by using (\ref{f2}), and then, we use the second in (\ref{f1}), to finally derive $\Pi_{1}(x,y)$. 

A key step to analyze the functional equation (\ref{f2}) is the careful examination of the algebraic curve defined by the kernel equation,
\begin{equation}
H(x,y):=D(x)R(x,y)-\tau\gamma xy=0.
\label{ker}
\end{equation}
It is easily seen that $H(x,y)$ is a polynomial of third degree in $x$, and of second degree in $y$. The study of $H(x,y)=0$ (see Appendix \ref{app1}) allows to continue the unknown functions
$\Pi_{0}(x, 0)$ , $\Pi_{0}(0,y)$ analytically outside the unit
disk, and to reduce their determination to a Dirichlet boundary value problem.

Let $\mathcal{C}_{x}=\{x\in\mathbb{C}:|x|=1\}$, $\mathcal{C}_{y}=\{y\in\mathbb{C}:|y|=1\}$, $\mathcal{D}_{x}=\{x\in\mathbb{C}:|x|\leq1\}$, $\mathcal{D}_{y}=\{y\in\mathbb{C}:|y|\leq1\}$, and denote by $\mathcal{U}^{+}$ (resp. $\mathcal{U}^{-}$) the interior (resp. the exterior) domain bounded by the contour $\mathcal{U}$. Then the following lemma provides information about the location of the zeros of the kernel $H(x,y)$.
\begin{lemma}\label{lem}
If $y\in\mathcal{C}_{y}$ (resp. $x\in\mathcal{C}_{x}$), $H(x,y)=0$ has a unique root, say $X_{0}(y)\in\mathcal{D}_{x}$ (resp. $Y_{0}(x)\in\mathcal{D}_{y}$).
\end{lemma}
\begin{proof}
See Appendix \ref{app1}
\end{proof}
By the implicit function theorem, we see that
the algebraic function $Y(x)$ (respectively $X(y)$) defined by $H(x,Y(x))=0$ (resp. $H(X(y),y)=0$) is analytic except at branch
points. Denote $X_{1}(y)$, $X_{2}(y)$ the other two in $x$, with $|X_{1}(1)|<|X_{2}(1)|$, by $Y_{1}(x)$ the other one in $y$.
\begin{lemma}\label{lem2}
The algebraic function $Y(x)$, defined by
$H(x,Y(x))=0$, has six real positive branch points, and two of them, say $x_{1}$, $x_{2}$, are such that $0=x_{1}<x_{2}<1$. 
\end{lemma}
\begin{proof}
See Appendix \ref{app1}.
\end{proof}

Define the image contour, $\mathcal{L}=Y_{0}([\overrightarrow{\underleftarrow{0,x_{2}}}])$, where $[\overrightarrow{\underleftarrow{u,v}}]$ stands for the contour traversed from $u$ to $v$ along the upper edge of the slit $[u,v]$ and then back to $u$ along the lower edge of the slit. The following lemma shows that the mapping $Y(x)$,  $x\in[0,x_{2}]$, gives rise to the smooth and closed contour $\mathcal{L}$.
\begin{remark}
The study of $H(x,y)=0$ with respect to $x$ is slightly more difficult, but allows to also show that the algebraic function $X(y)$ has also two real and non-negative branch points inside $\mathcal{D}_{y}$, say $0\leq y_{1}<y_{2}<1$. Similarly, for $y\in[y_{1},y_{2}]$, $X(y)$ lies on a closed contour $\mathcal{M}$. Further details are omitted due to space limitations, and mainly due to the fact that our main contribution relies on the use of PSA method.    
\end{remark}
\begin{lemma}
For $x\in[0,x_{2}]$, the algebraic function $Y(x)$ lies on a closed contour $\mathcal{L}$, which is symmetric with respect to the real line and such that
\begin{displaymath}
|y|^{2}=\frac{(1-p)\nu_{2}}{p\nu_{1}}x,\,\,|y|^{2}\leq\frac{(1-p)\nu_{2}}{p\nu_{1}}x_{2}.
\end{displaymath}
\end{lemma}
\begin{proof}
Follows directly from the fact that $\Delta(x)$ is negative for
$x\in(0,x_{2})$ (see Appendix \ref{app1}).
\end{proof}
\subsection{A boundary value problem for $\Pi_{0}(0,y)$}
For $x\in\mathcal{C}_{x}$, $y=Y_{0}(x)$ we obtain
\begin{equation}
(1-p)\Pi_{0}(x,0)-p\Pi_{0}(0,Y_{0}(x))+\frac{C(x,Y_{0}(x))}{F(x,Y_{0}(x))}\Pi_{0}(0,0)=0,
\label{fwq}
\end{equation}
where $F(x,y)=\frac{A(x,y)}{1-p}=-\frac{B(x,y)}{p}$. For $y\in\mathcal{C}_{y}$, $x=X_{0}(y)$
\begin{equation}
(1-p)\Pi_{0}(X_{0}(y),0)-p\Pi_{0}(0,y)+\frac{C(X_{0}(y),y)}{F(X_{0}(y),y)}\Pi_{0}(0,0)=0.
\label{fw}
\end{equation}
Equation (\ref{fw}) implies that
\begin{displaymath}
\Pi(0,y)=\frac{1-p}{p}\Pi_{0}(X_{0}(y),0)-\frac{C(X_{0}(y),y)}{F(X_{0}(y),y)}\Pi_{0}(0,0),
\end{displaymath}
which is a meromorphic function\footnote{Its possible poles are the zeros of $F(X_{0}(y),y)$ in $\mathcal{L}^{+}\cap\mathcal{C}_{y}^{-}$.}. For $y\in \mathcal{L}^{+}\cap\mathcal{C}_{y}^{-}$, $|X_{0}(y)|<1$, and thus, we can construct analytic continuation of $\Pi_{0}(0,y)$ for all $y\in\zeta_{y}:=\mathcal{L}^{+}\cap\mathcal{C}_{y}^{-}$\footnote{Note that $\zeta_{y}$ is an empty set when $(1-p)\nu_{2}\leq p\nu_{1}$, since in such a case $\mathcal{L}$ lies entirely inside the unit circle.}.

Moreover, as $Y_{0}(x)$ is analytic in $\mathcal{D}_{x}-[0,x_{2}]$ \cite{fay}, $\Pi_{0}(0,Y_{0}(x))$ is meromorphic in $\mathcal{D}_{x}-[0,x_{2}]$, and from (\ref{fwq})
\begin{equation}
\Pi_{0}(x,0)=\frac{p}{1-p}\Pi_{0}(0,Y_{0}(x))+\frac{C(x,Y_{0}(x))}{F(x,Y_{0}(x))}\Pi_{0}(0,0)=0,\,x\in\mathcal{D}_{x}-[0,x_{2}].
\end{equation}
We therefore have the relation (\ref{fwq}), not only for $x\in\mathcal{C}_{x}$ but also for $x\in\mathcal{D}_{x}-[0,x_{2}]$, and by continuity, for $x\in\mathcal{D}_{x}-[0,x_{2}]$ too. Since $\Pi_{0}(x,0)$ is real for $x\in[0,x_{2}]$, we obtain
\begin{displaymath}
Re[i\Pi(0,Y_{0}(x))]=Im[\Pi_{0}(0,0)\frac{C(x,Y_{0}(x))}{F(x,Y_{0}(x))}],\,x\in[0,x_{2}],
\end{displaymath}
or equivalently,
\begin{equation}
Re[i\Pi(0,y)]=c(y):=Im[\Pi_{0}(0,0)\frac{C(|y|^2p\nu_{1}/(1-p)\nu_{2},y)}{F(|y|^2p\nu_{1}/(1-p)\nu_{2},y)}],\,y\in\mathcal{L}.
\label{bv}
\end{equation}

Thus our problem is reduced to the determination of a function, which is regular for $y\in\mathcal{L}^{+}$, continuous in $\mathcal{L}^{+}\cup\mathcal{L}$ satisfying the boundary condition (\ref{bv}). A standard way to solve this Riemann-Hilbert boundary value problem is to conformally transformed it on the unit circle \cite{fay,van} by introducing the conformal mappings $z=\gamma(y):\mathcal{L}^{+}\to \mathcal{C}_{y}^{+}$, and its inverse $y=\gamma_{0}(z):\mathcal{C}_{y}^{+}\to\mathcal{L}^{+}$\footnote{See Appendix \ref{app2} for details on the numerical derivation of the conformal mappings}. 

Then, we have the following problem: Find a function $T(z)=\Pi_{0}(0,\gamma_{0}(z))$ regular for $z\in \mathcal{C}_{z}^{+}$, and continuous for $z\in\mathcal{C}_{z}\cup \mathcal{C}_{z}^{+}$ such that, $Re(iT(z))=c(\gamma_{0}(z))$, $z\in\mathcal{C}$. Its solution (see \cite{ga}) is given by
\begin{equation}
\Pi_{0}(0,y)=-\frac{1}{2\pi}\int_{\mathcal{C}_{z}}c(\gamma_{0}(z))\frac{z+\gamma(y)}{z-\gamma(y)}\frac{dz}{z}+K,\,y\in\mathcal{C}_{y}\cup\mathcal{C}_{y}^{+},
\label{sol}
\end{equation}
where $K$ is a constant. Then, (\ref{fwq}) can be used to obtain $\Pi_{0}(x,0)$, $x\in\mathcal{C}_{x}$ \footnote{Note that a similar analysis can be also performed to obtain $\Pi_{0}(x,0)$ in terms of another Riemann-Hilbert boundary value problem}. For $x\in\mathcal{C}_{x}^{+}$, $\Pi_{0}(x,0)$ can be derived by using the Cauchy's formula, yielding
\begin{displaymath}
\Pi_{0}(x,0)=\frac{1}{2\pi}\int_{C_{y}}\frac{V(y)}{y-x}dy,\,x\in\mathcal{C}_{x}^{+},
\end{displaymath}
where 
\begin{displaymath}
V(y)=\frac{1-p}{p}\Pi_{0}(X_{0}(y),0)-\frac{C(X_{0}(y),y)}{F(X_{0}(y),y)}\Pi_{0}(0,0),\,y\in\mathcal{C}_{y}.
\end{displaymath}

Using (\ref{f2}) we obtain $\Pi_{0}(x,y)$. Then, using the second in (\ref{f1}) we obtain $\Pi_{1}(x,y)$ and all unknowns are fully specified. 
\section{Numerical Results}\label{sec:num}
We now compare the PSA to the exact results derived by the BVP approach and investigate the influence of some parameters on the mean queue lengths. 

Figure \ref{fig1} depicts the approximations (\ref{pert}) as a function of $p$ for increasing values of $M$. Set $\lambda_{0}=1$, $\lambda_{1}=0.5$, $\tau=4$, $\gamma=2$, $\nu_{1}=4$, $\nu_{2}=5$. The horizontal lines ($M = 0$) equal the
values for the tandem system with priority for the second queue. Figure \ref{fig1} confirms that
the PSA approximations are accurate for $p$ close to 0, and clearly, more
terms provide larger regions for $p$ where the accuracy is good. 
\begin{figure}[ht!]
\centering
\includegraphics[scale=0.6]{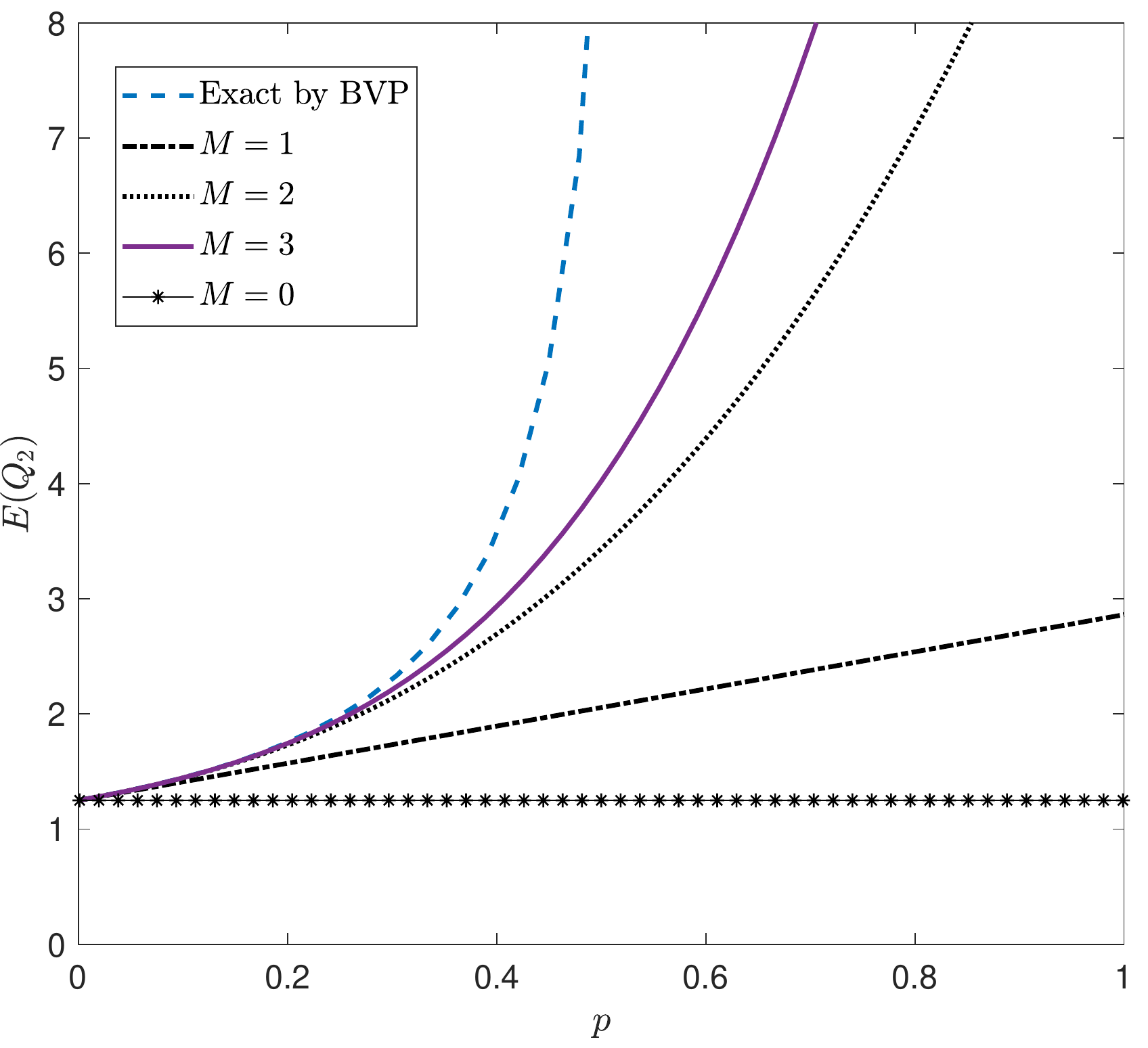}
\caption{Truncation approximations for $\lambda_{0}=1$, $\lambda_{1}=0.5$, $\tau=4$, $\gamma=2$, $\nu_{1}=4$, $\nu_{2}=5$}\label{fig1}
\end{figure}

In Figure \ref{fig11} ($\lambda_{0}=1$, $\lambda_{1}=0.5$, $M=3$, $\nu_{1}=4$, $\nu_{2}=5$), we can observe that the increase in $\gamma$ will definitely increase $E(Q_{2})$, since the system switches to the setup mode more frequently.
\begin{figure}[ht!]
\centering
\includegraphics[scale=0.6]{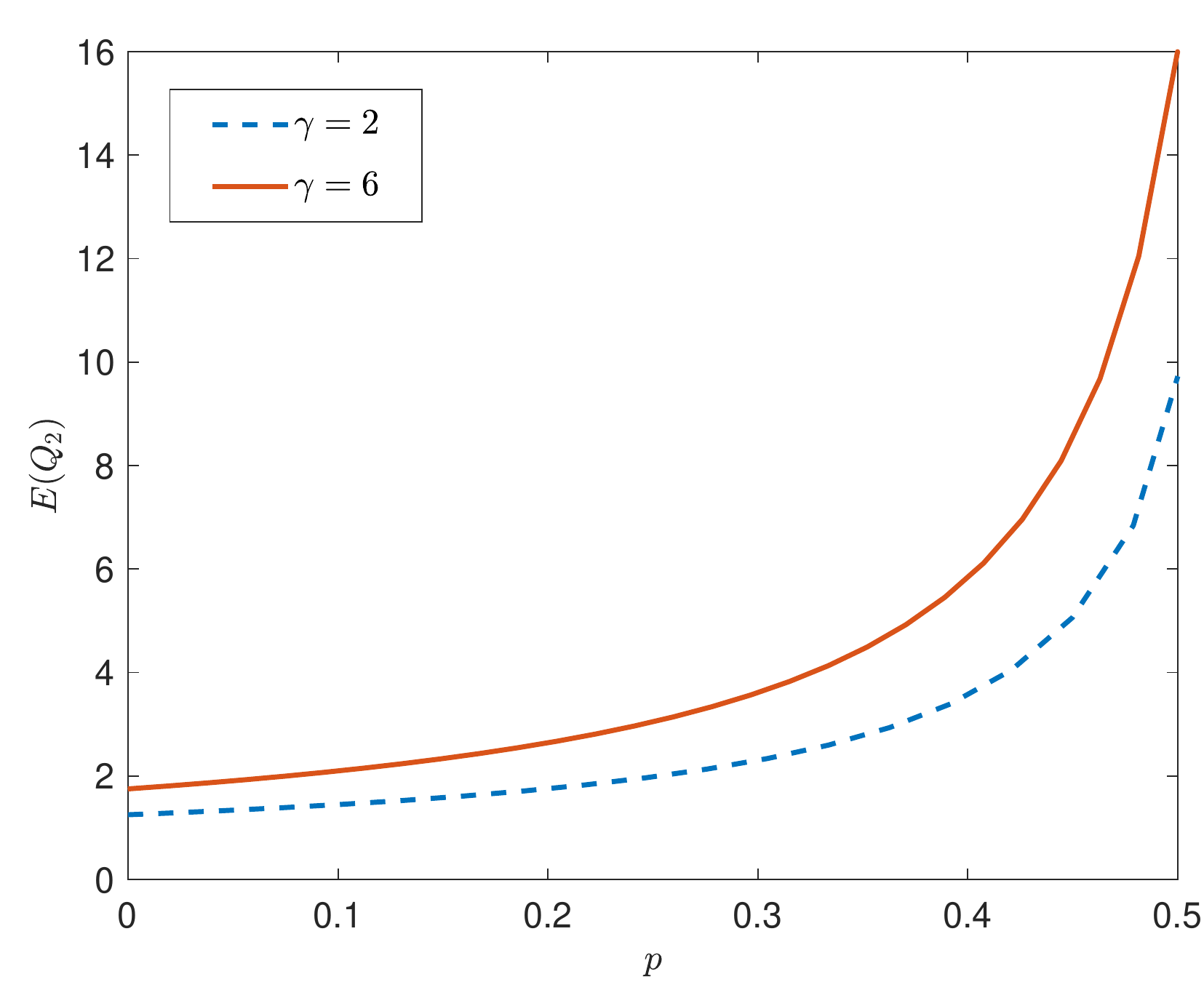}
\caption{Truncation approximations for $\lambda_{0}=1$, $\lambda_{1}=0.5$, $\tau=4$, $M=3$, $\nu_{1}=4$, $\nu_{2}=5$}\label{fig11}
\end{figure}
\section{Conclusion}
In this work, we investigated the stationary behaviour of an unreliable two-node tandem queue with coupled processors, which is described by a markov modulated RWQP. Based on the generating function approach, we applied the PSA method and obtained power series expansions of the pgfs of the stationary joint queue length distributions for each state of the network. With this result we shown the flexibility of the PSA approach to be applied in more complicated models. Moreover, we also obtained the corresponding pgfs with the aid of the theory of Riemann-Hilbert boundary value problems. By truncating the power series, we find good approximations for the expected number of customers especially when $p$ is close to 0, by comparing them with the ``exact" derivations through BVP.

In the future, we plan to expand our results to networks with more than two nodes, where the theory of BVPs cannot be applied, as well as to consider general routing among the nodes. Moreover, it could be also interesting to compare PSA method with other approximation techniques developed so far \cite{mitr,cha,rei,vin}. Moreover, it would be interesting to derive asymptotic estimates for the occurrence of large queue lengths due to the presence of failures.
\appendix
\section{Analysis of the kernel}\label{app1}
\textbf{Proof of Lemma \ref{lem}}
Let $u(x,y)=\lambda_{0}(1-x)+\nu_{1}p(1-\frac{y}{x})+\nu_{2}(1-p)(1-\frac{1}{y})+\gamma$. Note that
\begin{displaymath}
H(x,y)=0\Leftrightarrow xy\{D(x)u(x,y)-\gamma\tau\}=0.
\end{displaymath}
Using the principal value argument it is seen that the number of zeros of
$H(x,y)$ in $\mathcal{D}_{x}$ equals the number of zeros
of $D(x)u(x,y)$ in $\mathcal{D}_{x}$, which is equal to one.\hfill$\square$\vspace{2mm}\\
\textbf{Proof of Lemma \ref{lem2}} The branch points of the two-valued function $Y(x)$ are the zeros of the discriminant $\Delta(x)=x(f(x)+g(x))$ of $H(x,y)=0$, where
\begin{displaymath}
\begin{array}{rl}
f(x)=&D^{2}(x)[x(\lambda_{0}(1-x)+\gamma+p\nu_{1}+(1-p)\nu_{2})^2-4(p\nu_{1}+(1-p)\nu_{2})],\\
g(x)=&x\lambda_{1}(1-x)\gamma[\lambda_{1}(1-x)\gamma+2D(x)(\lambda_{0}(1-x)+p\nu_{1}+(1-p)\nu_{2})].
\end{array}
\end{displaymath}
Note $g(x)=0$ if and only if $x=0$, $x=1$, and
\begin{equation}
\begin{array}{r}
2\lambda_{0}\lambda_{1}x^{2}-x[2\lambda_{0}(2\lambda_{1}+\tau)+\lambda_{1}(\gamma+2(p\nu_{1}+(1-p)\nu_{2}))]\\+2(\tau+\lambda_{1})(\lambda_{0}+p\nu_{1}+(1-p)\nu_{2})+\lambda_{1}\gamma=0.
\end{array}\label{br}
\end{equation}
Let $x_{1}^{*}$, $x_{2}^{*}$ the zeros of (\ref{br}). Then,
\begin{displaymath}
\begin{array}{rl}
x_{1}^{*}x_{2}^{*}=&1+\frac{2(\tau+\lambda_{1})(p\nu_{1}+(1-p)\nu_{2})+\lambda_{1}\gamma+2\lambda_{0}\tau}{2\lambda_{0}\lambda_1}>1,\\
x_{1}^{*}+x_{2}^{*}=&2+\frac{\lambda_{1}\gamma+2\lambda_{1}(p\nu_{1}+(1-p)\nu_{2})+2\lambda_{0}\tau}{2\lambda_{0}\lambda_{1}}>2,
\end{array}
\end{displaymath}
which means that $x_{1}^{*},x_{2}^{*}>1$. Thus $g(x)=0$ has exactly two roots in $[0,1]$. By using Rouche's theorem, we can show that $D(x)=0$, i.e., $x(f(x)+g(x))=0$ has exactly two zeros in $[0,1]$, and one of them equals $x_{1}=0$.

\section{On the derivation of conformal mappings}\label{app2}
A detailed approach on how we can numerically obtain the conformal mappings is given in \cite{bv,van}. We summarized the basic steps. First, we need to represent $\mathcal{L}$ in polar coordinates, i.e., $\mathcal{L}=\{y:y=\rho(\phi)\exp(i\phi),\phi\in[0,2\pi]\}.$ Since $0\in\mathcal{L}^{+}$, for each $y\in\mathcal{L}$, we can have a relation between its absolute value and its real part, i.e., $|y|^{2}=m(Re(y))$. Given the angle $\phi$ of some point on $\mathcal{L}$, the real part of this point, say $\delta(\phi)$, is the solution of $\delta-\cos(\phi)\sqrt{m(\delta)}$, $\phi\in[0,2\pi].$ Since $\mathcal{L}$ is a smooth, egg-shaped contour, the solution is unique. Clearly, $\rho(\phi)=\frac{\delta(\phi)}{\cos(\phi)}$, and the parametrization of $\mathcal{L}$ is fully specified. Then, the mapping from $z\in \mathcal{C}_{z}^{+}$ to $y\in \mathcal{L}^{+}$, where $z = e^{i\phi}$ and $y= \rho(\psi(\phi))e^{i\psi(\phi)}$, satisfying $\gamma_{0}(0)=0$, $\gamma_{0}(z)=\overline{\gamma_{0}(z)}$ is uniquely determined by,
\begin{equation}
\begin{array}{rl}
\gamma_{0}(z)=&z\exp[\frac{1}{2\pi}\int_{0}^{2\pi}\log\{\rho(\psi(\omega))\}\frac{e^{i\omega}+z}{e^{i\omega}-z}d\omega],\,|z|<1,\\
\psi(\phi)=&\phi-\int_{0}^{2\pi}\log\{\rho(\psi(\omega))\}\cot(\frac{\omega-\phi}{2})d\omega,\,0\leq\phi\leq 2\pi,
\end{array}
\label{zx}
\end{equation}
i.e., $\psi(.)$ is uniquely determined as the solution of a Theodorsen integral equation with $\psi(\phi)=2\pi-\psi(2\pi-\phi)$. Due to the correspondence-boundaries theorem, $\gamma_{0}(z)$ is continuous in $\mathcal{C}_{z}\cup \mathcal{C}_{z}^{+}$. 
\section{On the analyticity of $\Pi_{j}(x,y)$ close to $p=0$}\label{appe3}
We focus only on the analyticity of $\Pi_{0}(x,y)$ in a neighborhood of $p=0$ by using a variant of the implicit function theorem on the functional equation (\ref{f2}). The analyticity of $\Pi_{1}(x,y)$, follows directly by the analyticity of $\Pi_{0}(x,y)$ from the second in (\ref{f1}). We follow the lines in \cite{walr,dim}, and use the implicit function theorem for Banach spaces (see Theorem 10.2.3, p. 272 in \cite{ban}). Define the mapping $f:S\subset\mathbb{C}\times B_{2}\to B_{3}\times\mathbb{C}$,
\begin{displaymath}
\begin{array}{l}
f(p,\Pi_{0})=[\Pi_{0}(x,y)H(x,y)-D(x)\{A(x,y)\Pi_{0}(x,0)\\+B(x,y)\Pi_{0}(0,y)+C(x,y)\Pi_{0}(0,0)\}, \Pi_{0}(1,1)-\frac{\tau}{\tau+\gamma}],
\end{array}
\end{displaymath}
where $S$ contains the point $(0,V_{0}^{(0)})$, $H$, $A$, $B$, $C$, are as in (\ref{ker}) and (\ref{ui}) respectively, $B_2$ be the Banach space comprising all bivariate analytic bounded
functions in $\mathbb{D}^{2}$, with $\mathbb{D}$ the open complex unit disk, and $B_3$ be the Banach
space comprising all trivariate analytic bounded functions in $\mathbb{D}^{3}$ that have a
limit of 0 for the first two arguments going to 1.

Since $H$, $A$, $B$, $C$ are bounded analytic functions in $\mathbb{D}^{3}$, and since $f$ is affine in $\Pi_{0}$ and $p$, it is easily seen that $f$ is $r$-times continuously differentiable for all $r$. Note also that $f (0,V_{0}^{(0)}) = [0, 0]$.
Then, the (Banach space) derivative of $f$ at the point $(0,V_{0}^{(0)})$ \cite{ban} equals
\begin{displaymath}
\begin{array}{r}
df(0,V_{0}^{(0)})=[\Pi_{0}(x,y)H(x,y)-D(x)\{A(x,y)\Pi_{0}(x,0)\\+C(x,y)\Pi_{0}(0,0)\}, \Pi_{0}(1,1)].
\end{array}
\end{displaymath}
We need to show that this mapping is a homeomorphism. Indeed,
\begin{enumerate}
\item $df (0,V_{0}^{(0)})$ is a continuous mapping for the same reasons
that the mapping $f$ itself is continuous.
\item For given $\Pi_{0}^{(1)}$, $\Pi^{(2)}_{0}$, let $df (0,V_{0}^{(0)})(\Pi_{0}^{(1)})=df (0,V_{0}^{(0)})(\Pi_{0}^{(1)})$. Then,
\begin{displaymath}
\begin{array}{c}
[\Pi_{0}^{(1)}(x,y)-\Pi_{0}^{(2)}(x,y)]H(x,y)\\-D(x)\{A(x,y)(\Pi_{0}^{(1)}(x,0)-\Pi_{0}^{(2)}(x,0))+C(x,y)(\Pi_{0}^{(1)}(0,0)-\Pi_{0}^{(2)}(0,0))\}=0,\\
\Pi_{0}^{(1)}(1,1)-\Pi_{0}^{(2)}(1,1)=0.
\end{array}
\end{displaymath}
or equivalently $f(0,\Pi_{0}^{(1)}-\Pi_{0}^{(2)})=(0,-\frac{\tau}{\tau+\gamma})$, which in turn has the zero solution as a unique solution \cite{asm}, and thus $\Pi^{(0)}_{1}=\Pi^{(0)}_{2}$ so that $df (0,V_{0}^{(0)})$ is injective.
\item To show that $df (0,V_{0}^{(0)})$ is surjective, we solve the $df (0,V_{0}^{(0)})(\Pi_{0}) = (g,c)$ with $g$ a bivariate analytic bounded function in $\mathbb{D}^{2}$ with limit 0 for its arguments going to 1, and $c$ a complex number. The solution is
\begin{displaymath}
\begin{array}{l}
\Pi_{0}(x,y)\\=\frac{g(x,y)A(x,Y_{0}(x))-g(x,Y_{0}(x))A(x,y)+\Pi^{(0)}(0,0)[C(x,y)A(x,Y_{0}(x))-C(x,Y_{0}(x))A(x,y)]}{H(x,y)A(x,Y_{0}(x))}.
\end{array}
\end{displaymath}
\item The $\Pi_{0}$ obtained previously equals $(df (0,V_{0}^{(0)}))^{-1}$, which is readily seen that it is continuous.
\end{enumerate}
Thus, $\Pi_{0}\to df(0,V_{0}^{(0)})(\Pi^{(0)})$ is a linear homeomorphism and using Theorem 10.2.3 in \cite{ban}, $\Pi_{0}(x,y)$ is $r$-times differentiable at $p=0$. Having this result, and using the second in (\ref{f1}), $\Pi_{1}(x,y)$ is also $r$-times differentiable at $p=0$.
%
%
\bibliographystyle{splncs04}
\bibliography{paper.bbl}

\end{document}